\documentclass[11pt,a4paper]{article}
\usepackage[T1]{fontenc}
\usepackage[utf8]{inputenc}
\usepackage[english]{babel}
\usepackage{lmodern,microtype}
\usepackage[margin=27mm]{geometry}
\usepackage{amsmath,amssymb,amsthm,mathtools}
\usepackage{booktabs,array,longtable}
\usepackage{enumitem}
\usepackage{xurl}
\usepackage{hyperref}
\usepackage{listings}
\hypersetup{hidelinks,pdftitle={A Divisor-Weighted Golomb Sequence: Zeta Asymptotics and Arithmetic Fluctuations},pdfauthor={Marco Mantovanelli}}
\setlist{nosep}
\newtheorem{theorem}{Theorem}[section]
\newtheorem{proposition}[theorem]{Proposition}
\newtheorem{lemma}[theorem]{Lemma}
\newtheorem{corollary}[theorem]{Corollary}
\theoremstyle{definition}
\newtheorem{definition}[theorem]{Definition}
\newtheorem{remark}[theorem]{Remark}
\newcommand{\N}{\mathbb N}
\newcommand{\E}{\mathbb E}
\newcommand{\Prob}{\mathbb P}
\newcommand{\ind}{\mathbf 1}
\newcommand{\Rea}{\operatorname{Re}}
\newcommand{\lcm}{\operatorname{lcm}}
\newcommand{\ph}{\varphi}
\DeclareMathOperator{\Var}{Var}
\numberwithin{equation}{section}
\title{\bfseries A Divisor-Weighted Golomb Sequence:\\
Zeta Asymptotics and Arithmetic Fluctuations}
\author{Marco Mantovanelli\\[0.35em]
\normalsize Independent Researcher, Germany\\
\normalsize\href{mailto:marco@mantovanelli.de}{\texttt{marco@mantovanelli.de}}\\
\normalsize\href{https://orcid.org/0009-0002-0631-293X}{ORCID: 0009-0002-0631-293X}}
\date{}

\begin{document}
\maketitle
\vspace{-1.5em}
\begin{abstract}
We study a divisor-weighted version of Golomb's self-description:
the nondecreasing sequence $A$ in which each positive integer $m$ occurs
$\sum_{d\mid m}A(d)$ times. This rule has a unique solution, and
\[
 A(n)\sim Cn^{\varphi-1},\qquad
 C=\left(\frac{\varphi}{\zeta(\varphi)}\right)^{2-\varphi},
 \qquad \varphi=\frac{1+\sqrt5}{2}.
\]
The golden-ratio exponent of Golomb's sequence survives, while divisor
weighting changes the leading constant.

We prove the asymptotic without assuming regular variation, by a
contraction of upper and lower power envelopes. The argument applies
to every nonnegative integer Dirichlet-convolution kernel $w$ with
$w(1)=1$ and $\sum_q w(q)q^{-\varphi}<\infty$. For the divisor kernel,
a quantitative version gives relative error
$O_\gamma((\log n)^{-\gamma})$ for every $0<\gamma<1$.

Although the global growth is smooth, the normalized run lengths retain
arithmetic fluctuations. A uniform comparison with the divisor sum
$\sum_{d\mid m}d^{1-\varphi}$ transfers them to an explicit random Euler product.
We prove that its law is singular continuous with full support
$[1,\infty)$, determine all real Mellin moments, obtain the sharp
maximal order of the fluctuations, and show that sampling sequence
positions instead of block labels produces the size-biased law.
\end{abstract}



\section{Introduction}\label{sec:introduction}

A self-describing sequence turns its own values into instructions for
how long those values should persist. In Golomb's classical sequence,
the value $m$ occurs exactly as many times as the $m$th term specifies.
The sequence is therefore the inverse of a counting function that it
generates itself. This feedback explains why a simple integer rule can
produce an irrational growth exponent.

Here the feedback also carries arithmetic information. We replace the
single term prescribing a block length by the sum of the terms at all
divisors of the block label. Prime labels then have only two
contributions, whereas labels with many small prime factors can have
much longer blocks. It is not evident that these irregular local
weights should permit a single global power law. The central result
is that they do: the exponent remains Golomb's, and the arithmetic
weight enters through an explicit value of the Riemann zeta function.
The irregularity does not disappear; it survives in the distribution,
extremes, and correlations of the normalized block lengths.

\subsection{The sequence and its principal asymptotic}

Let $\N=\{1,2,\ldots\}$. We seek a nondecreasing sequence
$A:\N\to\N$ in which every positive integer occurs, and write
\begin{equation}\label{eq:frequency}
 B(m)=\sum_{d\mid m}A(d),\qquad
 F(0)=0,\qquad F(m)=\sum_{j\le m}B(j).
\end{equation}
The self-description is the inverse relation
\begin{equation}\label{eq:inverse}
 A(n)=m\quad\Longleftrightarrow\quad F(m-1)<n\le F(m).
\end{equation}
Thus $n$ is a \emph{position}, $m$ a \emph{value} or \emph{block label},
$B(m)$ a \emph{block length}, and $F(m)$ a \emph{block endpoint}.
We also call $F$ the self-generated clock. The first terms are
\begin{align*}
 A(n):\;&1,2,2,2,3,3,3,4,4,4,4,4,5,5,5,5,6,6,\ldots,\\
 B(m):\;&1,3,3,5,4,8,4,9,7,10,5,14,6,11,11,14,\ldots.
\end{align*}
The block lengths need not be monotone; for instance,
$B(6)=8>B(7)=4$.\footnote{The divisor-weighted sequence studied here
is listed in the OEIS as A400220 \cite{oeis400220}. It differs from the
classical Golomb sequence A001462 \cite{oeis} at the fourth term.}

Throughout the paper, put
\begin{equation}\label{eq:constants}
 \ph=\frac{1+\sqrt5}{2},\qquad
 \alpha=\ph-1=\frac1\ph,\qquad
 K=\frac{\zeta(\ph)}{\ph},\qquad C=K^{-\alpha^2}.
\end{equation}
In particular, $\alpha^2=1-\alpha=2-\ph$ and
$C=0.883418998410228445\ldots$.

\begin{theorem}[Global growth]\label{thm:main}
Equations \eqref{eq:frequency}--\eqref{eq:inverse} define a unique
nondecreasing sequence $A:\N\to\N$ that assumes every positive integer.
Its increments belong to $\{0,1\}$, and
\begin{equation}\label{eq:mainasymptotic}
 A(n)\sim Cn^\alpha,\qquad
 F(m)\sim CKm^\ph=C^{-\ph}m^\ph.
\end{equation}
For every fixed $0<\gamma<1$,
\begin{equation}\label{eq:rateintro}
 A(n)=Cn^\alpha\bigl(1+O_\gamma((\log n)^{-\gamma})\bigr).
\end{equation}
\end{theorem}

The existence proof is causal: the data needed for a new block are
already present in earlier blocks. The asymptotic proof starts from
two-sided power bounds, not a power-law assumption. If $\ell$ and $L$
are the lower and upper asymptotic coefficients, inversion of the
clock yields $L/\ell\le(L/\ell)^\alpha$. Since $0<\alpha<1$, these
envelopes collapse. This is the decisive self-consistency step.

The same mechanism works for
$B_w(m)=\sum_{d\mid m}w(m/d)A_w(d)$ whenever $w$ is nonnegative and
integer-valued, $w(1)=1$, and
$W(\ph)=\sum_q w(q)q^{-\ph}<\infty$.
Theorem~\ref{thm:kernel} gives the coefficient
$C_w=(\ph/W(\ph))^{2-\ph}$.
The Dirichlet identity kernel recovers Golomb's sequence; the constant
kernel $w=1$ gives \eqref{eq:frequency}. Thus the zeta factor is explained
by a general kernel principle, rather than fitted to numerical data.

\subsection{From smooth growth to arithmetic fluctuations}

Write $\sigma_t(m)=\sum_{d\mid m}d^t$ and $R_m=B(m)/(Cm^\alpha)$.
The second principal result connects these normalized block lengths
to a classical divisor function and distinguishes the two natural
sampling procedures.

\begin{theorem}[Arithmetic fluctuations]\label{thm:arithmetic}
For every fixed $0<\gamma<1$, uniformly as $m\to\infty$,
\begin{equation}\label{eq:Bintro}
 B(m)=C\sigma_\alpha(m)
 \bigl(1+O_\gamma((\log m)^{-\gamma})\bigr).
\end{equation}
Let $(V_p)_p$ be independent random variables indexed by the primes,
with
\[
 \Prob(V_p=j)=(1-p^{-1})p^{-j}\qquad(j=0,1,\ldots),
 \qquad Z=\prod_p\left(\sum_{j=0}^{V_p}p^{-j\alpha}\right).
\]
This product is finite almost surely, with $\E Z=\zeta(\ph)$.
For every bounded continuous $h:[0,\infty)\to\mathbb R$,
\begin{align*}
 \lim_{M\to\infty}\frac1M\sum_{m\le M}h(R_m)
 &=\E[h(Z)],\\
 \lim_{N\to\infty}\frac1N\sum_{n\le N}h(R_{A(n)})
 &=\frac{\E[Zh(Z)]}{\zeta(\ph)}.
\end{align*}
Thus uniform sampling of block labels gives the Euler-product law,
whereas uniform sampling of sequence positions gives its size-biased law.
\end{theorem}

The uniform estimate is proved in Proposition~\ref{prop:uniform};
the two limiting laws are established in
Theorems~\ref{thm:distribution} and~\ref{thm:sizebias}.
The comparison $R_m\sim\sigma_{-\alpha}(m)$ holds along \emph{all}
integers, not merely on average. In particular,
\begin{equation}\label{eq:liminfintro}
 \liminf_{m\to\infty}R_m=1,\qquad
 \limsup_{m\to\infty}R_m=\infty,
\end{equation}
and the extremes satisfy
\begin{equation}\label{eq:maxintro}
 \log\max_{m\le x}R_m
 \sim\frac{(\log x)^{1-\alpha}}{(1-\alpha)\log\log x}.
\end{equation}
The Euler-product law is singular continuous with full support
$[1,\infty)$ and finite Mellin moments of every real order.
These refinements and the maximal order describe the fluctuations
identified by Theorem~\ref{thm:arithmetic}.

The point is not that the divisor function itself is new.
After \eqref{eq:Bintro}, several results in
Sections~\ref{sec:fluctuations}--\ref{sec:correlations} are classical
multiplicative-number-theoretic consequences transferred through a
uniform comparison. They are not presented as new general results
about divisor functions. The contribution is the bridge from a
nonmultiplicative, self-referential sequence to that arithmetic.
In addition, position sampling is governed by the sequence's own
clock: identifying its size-biased limit requires the weighted
summation and incomplete-block arguments, not just a limiting law
for $\sigma_{-\alpha}(m)$.

\subsection{Relation to earlier work and organization}

Golomb introduced the unweighted self-description in Problem~5407
\cite{golomb}. The published solutions contain a heuristic argument by
Marcus \cite{marcus} and a proof by Fine \cite{fine}. Vardi studied the error term
\cite{vardi}; the subsequent work of P\'etermann and R\'emy
\cite{petermannremy} describes a finer antiperiodic oscillation on the
$\log\log n$ scale.
These results concern the unweighted rule $B(m)=A(m)$.
Our kernel theorem includes its first-order asymptotic, but does not
subsume its sharper error-term theory.

Recent variants illustrate other ways to change self-description.
Cloitre's almost Golomb sequences impose finite-window frequency
conditions \cite{cloitre}; Claes and Miyamoto study Golombic and Levine
sequence operators \cite{claes}. In contrast, our multiplicities are
formed by Dirichlet convolution, with all positive block labels retained.
The recurrence in Appendix~\ref{app:structure} is value-dependent, but
the construction is not a fixed-order Hofstadter sum recursion.

The distributional viewpoint belongs to probabilistic number theory:
the logarithm of a positive multiplicative function is additive, and
its prime components suggest independent random variables
\cite{erdoswintner}. We supply direct truncation and moment proofs for
the particular divisor function needed here. Size bias is a general
sampling principle \cite{arratia}; the additional issue here is the
self-generated weighting and the passage through an incomplete final
block. Classical divisor identities and the prime number theorem are
used in their standard forms \cite{dlmf}.

Section~\ref{sec:construction} establishes causality.
Section~\ref{sec:kernel} proves the kernel theorem, and
Section~\ref{sec:rate} gives the quantitative refinement.
Section~\ref{sec:fluctuations} develops the divisor comparison, maximal
order, and Euler-product law.
Section~\ref{sec:sampling} studies position sampling and arithmetic
densities; Section~\ref{sec:correlations} treats dependence between
blocks. Section~\ref{sec:conclusion} closes with further questions.
Additional exact identities and the reproducible finite computations
are collected in the appendices.

\section{Causal construction}\label{sec:construction}

\begin{proposition}[Existence and uniqueness]\label{prop:construction}
There is exactly one sequence satisfying
\eqref{eq:frequency}--\eqref{eq:inverse}. Its initial two blocks are
$1$ and $2,2,2$. For every $m\ge3$, the values $A(d)$ with $d\mid m$
are already known before block $m$ is appended.
\end{proposition}
\begin{proof}
Set $A(1)=1$. The first block has length $B(1)=1$. The next position
must therefore carry the value $2$, which fixes $A(2)=2$ and gives
$B(2)=A(1)+A(2)=3$. Hence $F(2)=4$.

Suppose the blocks up to $m-1$ have been constructed. Every block
$j\ge2$ has length at least $A(j)\ge2$, so
\begin{equation}\label{eq:causality}
 F(m-1)\ge1+2(m-2)=2m-3\ge m\qquad(m\ge3).
\end{equation}
Thus position $m$, and all its divisor positions, are already defined.
The finite positive integer $B(m)=\sum_{d\mid m}A(d)$ can be calculated,
and appending exactly $B(m)$ copies of $m$ extends the sequence without
altering any previous term. The resulting endpoints tend to infinity,
so every position is eventually assigned. Each block is nonempty and
finite. This proves existence, slow growth, and the occurrence of every
positive integer.

In any solution the initial two blocks are forced. The same argument
then forces every subsequent block, proving uniqueness.
\end{proof}

Interchanging two finite sums gives the useful endpoint formula
\begin{equation}\label{eq:floor}
 F(m)=\sum_{d\le m}A(d)\left\lfloor\frac md\right\rfloor.
\end{equation}
In particular, $A(F(m))=m$ and $A(F(m)+1)=m+1$.

The divisor identity also retains exact arithmetic information.

\begin{proposition}[Primes and divisibility chains]\label{prop:primecode}
If $m\mid n$ and $m<n$, then $B(m)<B(n)$. Moreover, for every $m\ge2$,
\begin{equation}\label{eq:primecode}
 m\text{ is prime}\quad\Longleftrightarrow\quad B(m)=A(m)+1.
\end{equation}
If $m$ is composite, then $B(m)-A(m)\ge3$; in particular, the difference
$B(m)-A(m)$ is never $2$.
\end{proposition}
\begin{proof}
When $m$ properly divides $n$, every divisor of $m$ is a divisor of $n$,
and $n$ itself occurs only in the latter divisor sum. Hence
\[
 B(n)-B(m)=\sum_{\substack{d\mid n\\d\nmid m}}A(d)>0.
\]
For $m\ge2$,
\[
 B(m)-A(m)-1=
 \sum_{\substack{d\mid m\\1<d<m}}A(d).
\]
The sum is empty exactly when $m$ is prime. If it is nonempty, each of
its terms is at least $2$, proving the stronger composite bound.
\end{proof}

\section{A Dirichlet-kernel theorem}\label{sec:kernel}
The growth proof becomes clearer in a class that includes both the
classical Golomb sequence and the divisor-weighted sequence.

\begin{definition}
Let $w:\N\to\mathbb Z_{\ge0}$ satisfy $w(1)=1$. A
\emph{$w$-weighted self-describing sequence} has multiplicities
\begin{equation}\label{eq:kernel}
 B_w(m)=\sum_{d\mid m}w(m/d)A_w(d),\qquad
 F_w(m)=\sum_{j\le m}B_w(j),
\end{equation}
and is the increasing-block inverse of $F_w$ as in
\eqref{eq:inverse}. Write
\begin{equation}\label{eq:W}
 W(s)=\sum_{q\ge1}\frac{w(q)}{q^s}
\end{equation}
where the series converges.
\end{definition}

\begin{theorem}[Finite Dirichlet mass]\label{thm:kernel}
For every such $w$, the weighted sequence exists uniquely. If
$W(\ph)<\infty$, then
\begin{equation}\label{eq:kernelasymptotic}
 A_w(n)\sim C_w n^\alpha,\qquad
 C_w=\left(\frac{\ph}{W(\ph)}\right)^{\alpha^2},
\end{equation}
and
\begin{equation}\label{eq:Fkernel}
 F_w(m)\sim \frac{C_w W(\ph)}{\ph}m^\ph.
\end{equation}
Neither monotonicity nor boundedness of the multiplicities is assumed.
\end{theorem}

\subsection{Existence and comparison sums}
For $w(1)=1$, block $1$ still consists of a single $1$, and
$B_w(2)=2+w(2)\ge2$. Every later block has length at least $A_w(m)$.
Since the first block has length one, $A_w(j)\ge2$ for $j\ge2$.
Thus $F_w(m-1)\ge1+2(m-2)\ge m$ for $m\ge3$, so every required
divisor position is already known. This proves existence and uniqueness
as in Proposition~\ref{prop:construction}. It also proves
\begin{equation}\label{eq:strict}
 A_w(n)<n\qquad(n\ge3).
\end{equation}

For $x\ge1$, put
\begin{equation}\label{eq:Tw}
 T_w(x)=\sum_{qd\le x}w(q)d^\alpha.
\end{equation}
The inequalities
\begin{equation}\label{eq:Tbounds}
 \frac{\lfloor x\rfloor^\ph}{\ph}
 \le T_w(x)\le W(\ph)x^\ph
\end{equation}
follow from the $q=1$ summand, integration, and
$\sum_{d\le y}d^\alpha\le y^\ph$. Since
\[
 \sum_{d\le y}d^\alpha\sim\frac{y^\ph}{\ph},
\]
dominated convergence in the sum over $q$ gives
\begin{equation}\label{eq:Tlimit}
 T_w(x)\sim K_w x^\ph,
 \qquad K_w=\frac{W(\ph)}{\ph}.
\end{equation}
The domination is $w(q)q^{-\ph}$, which is summable by hypothesis.

We will also need a consequence that excludes a hidden contribution
from finitely many small divisor positions:
\begin{equation}\label{eq:wsmall}
 \sum_{q\le x}w(q)=o(x^\ph).
\end{equation}
Indeed, after division by $x^\ph$, the contribution from $q\le Q$ is
$O_Q(x^{-\ph})$, while the remaining contribution is at most
$\sum_{q>Q}w(q)q^{-\ph}$. Let $x\to\infty$ and then $Q\to\infty$.

\subsection{Two-sided power bounds without an ansatz}
\begin{lemma}\label{lem:envelope}
If $W(\ph)<\infty$, there is $M\ge2$, depending on $w$, such that
\begin{equation}\label{eq:envelope}
 M^{-1}n^\alpha\le A_w(n)\le Mn^\alpha
 \qquad(n\ge1).
\end{equation}
\end{lemma}
\begin{proof}
Choose $M$ large enough that
\begin{equation}\label{eq:Mchoice}
 M^{1-\alpha}\ge W(\ph)^\alpha,
 \qquad 1+(M\ph)^\alpha\le M.
\end{equation}
The first condition is exactly
$(MW(\ph))^{-\alpha}\ge M^{-1}$, which will close the lower induction
bound; the second closes the upper bound. Both hold for all sufficiently
large $M$ because $\alpha<1$. The initial cases $n=1,2$
are immediate. Suppose the bounds hold at every index below $n\ge3$,
and let $m=A_w(n)$. By \eqref{eq:strict}, all values $A_w(d)$ with
$d\le m$ are covered by the induction hypothesis. Hence
\[
 n\le F_w(m)\le M T_w(m)\le M W(\ph)m^\ph,
\]
which yields
\[
 m\ge(MW(\ph))^{-\alpha}n^\alpha\ge M^{-1}n^\alpha.
\]
For the upper bound use only the $q=1$ contribution:
\[
 n>F_w(m-1)\ge M^{-1}\sum_{d\le m-1}d^\alpha
 \ge\frac{(m-1)^\ph}{M\ph}.
\]
Consequently,
\[
 m<1+(M\ph n)^\alpha
 \le\bigl(1+(M\ph)^\alpha\bigr)n^\alpha\le Mn^\alpha.
\]
This completes a strong induction on positions, rather than a formal
induction on the still-unknown growth law.
\end{proof}

\subsection{Collapse of the asymptotic envelopes}
\begin{proof}[Proof of Theorem~\ref{thm:kernel}]
Existence has been established. Define
\[
 \ell=\liminf_{n\to\infty}\frac{A_w(n)}{n^\alpha},
 \qquad
 L=\limsup_{n\to\infty}\frac{A_w(n)}{n^\alpha}.
\]
Lemma~\ref{lem:envelope} gives $0<\ell\le L<\infty$.
For any $\varepsilon>0$, all sufficiently large $d$ satisfy
\[
 (\ell-\varepsilon)d^\alpha\le A_w(d)
 \le(L+\varepsilon)d^\alpha.
\]
For every fixed $D$, equation \eqref{eq:wsmall} gives explicitly
\[
 \sum_{d\le D}A_w(d)\sum_{q\le m/d}w(q)=o(m^\ph).
\]
The same estimate holds with $A_w(d)$ replaced by $d^\alpha$.
Thus omitting these finitely many divisor positions changes both
the clock and its comparison sum by $o(m^\ph)$.
Equation \eqref{eq:Tlimit} now implies
\begin{equation}\label{eq:Fenvelopes}
 \ell K_w\le\liminf_m\frac{F_w(m)}{m^\ph}
 \le\limsup_m\frac{F_w(m)}{m^\ph}\le LK_w.
\end{equation}
To make the inversion explicit, fix $0<\varepsilon<\ell K_w$.
For all sufficiently large $n$, with $m=A_w(n)$,
\[
 (\ell K_w-\varepsilon)(m-1)^\ph<n
 \le(LK_w+\varepsilon)m^\ph.
\]
The threshold is uniform because the power envelope gives $m\to\infty$.
Hence
\[
 (LK_w+\varepsilon)^{-\alpha}
 \le\frac{A_w(n)}{n^\alpha}
 <(\ell K_w-\varepsilon)^{-\alpha}+n^{-\alpha}.
\]
Taking lower and upper limits, then letting $\varepsilon\downarrow0$, gives
\begin{equation}\label{eq:contraction}
 \ell\ge(LK_w)^{-\alpha},\qquad
 L\le(\ell K_w)^{-\alpha}.
\end{equation}
Dividing the second bound by the first shows
\[
 1\le\frac L\ell\le\left(\frac L\ell\right)^\alpha.
\]
Because $0<\alpha<1$, this forces $L=\ell=c$. Both inequalities
in \eqref{eq:contraction} then force
$c=(cK_w)^{-\alpha}$. Solving gives
\[
 c=K_w^{-\alpha/(1+\alpha)}=K_w^{-\alpha^2}=C_w.
\]
Finally, substituting the resulting limit into
\eqref{eq:Fenvelopes} proves \eqref{eq:Fkernel}.
\end{proof}

\begin{remark}[Why the exponent of the constant matters]
Inverting a power law gives the relation
$c=(\ph/(cW(\ph)))^\alpha$, not
$c=(\ph/W(\ph))^\alpha$. The unknown leading coefficient occurs
on both sides. Its correct exponent is
$\alpha/(1+\alpha)=\alpha^2=2-\ph$.
\end{remark}

\subsection{Examples and a universality statement}
For $w(q)=1$, $W(\ph)=\zeta(\ph)$, proving the growth part of
Theorem~\ref{thm:main}. For the Dirichlet identity kernel
$w(q)=\ind_{\{q=1\}}$, we recover Golomb's constant
$\ph^{2-\ph}$. If $d_r$ denotes the coefficients of $\zeta(s)^r$,
with $d_0$ the identity kernel, then
\begin{equation}\label{eq:drfamily}
 B_r=d_r*A_r
 \quad\Longrightarrow\quad
 A_r(n)\sim
 \left(\frac{\ph}{\zeta(\ph)^r}\right)^{2-\ph}n^{\ph-1}
 \qquad(r=0,1,2,\ldots).
\end{equation}
Thus the same exponent persists through every finite number of divisor
convolutions. Another example is the squarefree kernel $w(q)=\mu(q)^2$,
for which
\begin{equation}\label{eq:squarefree}
 C_w=\left(\frac{\ph\zeta(2\ph)}{\zeta(\ph)}\right)^{2-\ph}.
\end{equation}
The Dirichlet-series identities used here are classical
\cite{dlmf}. The kernel theorem identifies how their values enter a
nonlinear self-describing fixed point. It does not apply to a kernel
with $W(\ph)=\infty$.

\section{A near-logarithmic convergence rate}\label{sec:rate}
We return to the divisor kernel $w=1$. The qualitative contraction from
Section~\ref{sec:kernel} can be made quantitative in two steps: first
control the error of the clock, then transfer it through inversion.
The constants in the estimates below are uniform in the position or
block label; they may depend on the fixed parameters.

\begin{lemma}[Comparison-sum remainder]\label{lem:Terror}
As $x\to\infty$,
\begin{equation}\label{eq:Terror}
 T_1(x)=Kx^\ph+O(x).
\end{equation}
\end{lemma}
\begin{proof}
The estimate
$\sum_{d\le y}d^\alpha=y^\ph/\ph+O(y^\alpha)$, valid for $y\ge1$, gives
\[
 T_1(x)=\frac{x^\ph}{\ph}\sum_{q\le x}q^{-\ph}
       +O\left(x^\alpha\sum_{q\le x}q^{-\alpha}\right).
\]
The error is $O(x)$ because $\alpha<1$. Extending the first sum to
infinity also costs $O(x)$.
\end{proof}

Write
\[
 u(n)=\frac{A(n)}{Cn^\alpha}-1,\qquad
 E(x)=\sup_{n\ge\max(1,\lceil x\rceil)}|u(n)|.
\]
The kernel theorem shows that $E$ is bounded, nonincreasing, and tends
to zero.

\begin{lemma}[Quantitative clock estimate]\label{lem:clock}
For every fixed $0<\theta<1$ there are constants $D_\theta,m_\theta>0$
such that, for every integer $m\ge m_\theta$,
\begin{equation}\label{eq:clockerror}
 \left|\frac{F(m)}{CKm^\ph}-1\right|
 \le E(m^\theta)+D_\theta m^{-\alpha(1-\theta)}.
\end{equation}
\end{lemma}
\begin{proof}
Split the error in the exact identity
\[
 F(m)=C T_1(m)+C\sum_{d\le m}d^\alpha u(d)\lfloor m/d\rfloor
\]
at $d=m^\theta$. Boundedness of $u$ gives, for the small divisors,
\[
 \left|\sum_{d\le m^\theta}d^\alpha u(d)\lfloor m/d\rfloor\right|
 \ll m\sum_{d\le m^\theta}d^{\alpha-1}
 \ll m^{1+\alpha\theta}.
\]
The remaining sum has absolute value at most $E(m^\theta)T_1(m)$.
After division by $Km^\ph$, Lemma~\ref{lem:Terror} proves the result.
The extra factor $1+O(m^{-\alpha})$ multiplying $E(m^\theta)$ is
absorbed into the remainder because $E$ is bounded.
\end{proof}

\begin{lemma}[Logarithmic contraction]\label{lem:logcontraction}
Suppose $E:[1,\infty)\to[0,\infty)$ is bounded and, for all sufficiently
large $x$,
\[
 E(x)\le qE(x^b)+Dx^{-\eta},
 \qquad 0<b,q<1,\quad D,\eta>0.
\]
If $\gamma>0$ and $qb^{-\gamma}<1$, then
$E(x)=O((\log x)^{-\gamma})$.
\end{lemma}
\begin{proof}
Set $J(t)=t^\gamma E(e^t)$ and $\rho=qb^{-\gamma}<1$. Above a fixed
threshold $T>0$,
\[
 J(t)\le\rho J(bt)+D t^\gamma e^{-\eta t}.
\]
Choose $k$ so that $b^kt\in[bT,T]$ and $b^jt>T$ for $j<k$.
Iterating gives
\[
 J(t)\le \rho^kJ(b^kt)
 +D\sum_{j=0}^{k-1}\rho^j(b^jt)^\gamma e^{-\eta b^jt}.
\]
Boundedness of $E$ controls $J$ on $[bT,T]$, while
$\sup_{s\ge bT}s^\gamma e^{-\eta s}<\infty$ bounds each factor in the
sum apart from $\rho^j$. The geometric sum is at most $(1-\rho)^{-1}$,
uniformly in $t$ and $k$. Thus $J$ is bounded.
\end{proof}

\begin{theorem}[Near-logarithmic rate]\label{thm:rate}
For every fixed $0<\gamma<1$,
\[
 \frac{A(n)}{Cn^\alpha}-1=O_\gamma((\log n)^{-\gamma}).
\]
\end{theorem}
\begin{proof}
Let $m=A(n)$. The power envelope gives
\[
 B(m)\ll\sigma_\alpha(m)
 =m^\alpha\sum_{q\mid m}q^{-\alpha}
 \le m^\alpha\sum_{q\le m}q^{-\alpha}\ll m.
\]
Since $0\le F(m)-n<B(m)$, replacing $F(m)$ by $n$ in
Lemma~\ref{lem:clock} changes the normalized error by $O(m^{-\alpha})$.
We may therefore write $n=CKm^\ph(1+v_n)$, where
\[
 |v_n|\le E(m^\theta)+O_\theta(m^{-\alpha(1-\theta)}).
\]
In particular $v_n\to0$. The normalization $C(CK)^\alpha=1$ gives
$u(n)=(1+v_n)^{-\alpha}-1$. For any fixed $q>\alpha$, the derivative of
$(1+v)^{-\alpha}$ is bounded in absolute value by $q$ on a sufficiently
small neighborhood of zero. Uniformly for all sufficiently large $n$,
\begin{equation}\label{eq:rawcontraction}
 |u(n)|\le qE(m^\theta)+O_{\theta,q}(m^{-\alpha(1-\theta)}).
\end{equation}

Fix $\gamma\in(0,1)$. Choose $\theta<1$ close enough to $1$,
$b<\alpha\theta$ close enough to $\alpha$, and $q>\alpha$ close enough
to $\alpha$, with $0<b,q<1$, so that
\begin{equation}\label{eq:qb}
 qb^{-\gamma}<1.
\end{equation}
Such choices exist because the limiting value is
$\alpha^{1-\gamma}<1$. The power envelope implies $m\asymp n^\alpha$,
hence eventually $m^\theta\ge n^b$. Using monotonicity of $E$ and taking
the supremum over integers $n\ge\lceil x\rceil$ in
\eqref{eq:rawcontraction} yields
\begin{equation}\label{eq:tailcontraction}
 E(x)\le qE(x^b)+Dx^{-\eta},\qquad
 \eta=\alpha^2(1-\theta)>0,
\end{equation}
for all sufficiently large real $x$. Here $D$ and the threshold
depend only on the chosen parameters and the fixed sequence.
Lemma~\ref{lem:logcontraction} proves the claim.
\end{proof}

The endpoint $\gamma=1$ is not obtained: the limiting contraction factor
in \eqref{eq:qb} then becomes $1$. This is the obstruction in the
present argument, not evidence that a $1/\log n$ relative bound is false.

\section{Arithmetic fluctuations and their limiting law}\label{sec:fluctuations}
\subsection{Uniform comparison with a divisor function}
\begin{proposition}[Uniform divisor comparison]\label{prop:uniform}
For every $0<\gamma<1$,
\begin{equation}\label{eq:uniform}
 B(m)=C\sigma_\alpha(m)
 \left(1+O_\gamma((\log m)^{-\gamma})\right).
\end{equation}
In particular, if
\begin{equation}\label{eq:RZ}
 R_m=\frac{B(m)}{Cm^\alpha},\qquad
 Z_m=\sigma_{-\alpha}(m)=\sum_{d\mid m}d^{-\alpha},
\end{equation}
then $R_m/Z_m\to1$.
\end{proposition}
\begin{proof}
The qualitative assertion follows by fixing $D$, applying
$A(d)=Cd^\alpha(1+o(1))$ uniformly for $d>D$, and observing that
the error from the finitely many $d\le D$ is $O_D(1)$, whereas
$\sigma_\alpha(m)\ge m^\alpha$.

For the stated rate choose $\delta=\alpha/(2\ph)>0$. Divisors
$d\le m^\delta$ contribute an error at most
\[
 O\left(\sum_{d\le m^\delta}d^\alpha\right)=O(m^{\alpha/2}),
\]
using the power envelope. On the remaining divisors,
Theorem~\ref{thm:rate} bounds the relative summand error by
$O_\gamma((\log m)^{-\gamma})$. Divide their sum by
$C\sigma_\alpha(m)$, and use
$\sigma_\alpha(m)=m^\alpha\sigma_{-\alpha}(m)$. The first part then
contributes $O(m^{-\alpha/2})=o((\log m)^{-\gamma})$.
\end{proof}

\begin{corollary}[Primes, prime powers, and primorials]\label{cor:primes}
For a fixed prime $p$ and a fixed positive integer $q$,
\begin{align}
 \lim_{\substack{r\to\infty\\r\ \mathrm{prime}}}R_r&=1,
       \label{eq:prime}\\
 \lim_{j\to\infty}R_{p^j}&=\frac1{1-p^{-\alpha}},
       \label{eq:primepower}\\
 \lim_{\substack{r\to\infty\\r\ \mathrm{prime},\ r\nmid q}}R_{qr}
       &=\sigma_{-\alpha}(q).\label{eq:fixedpart}
\end{align}
Along primorials $P(y)=\prod_{p\le y}p$, one has
\begin{equation}\label{eq:primorial}
 R_{P(y)}\sim\prod_{p\le y}(1+p^{-\alpha})\longrightarrow\infty.
\end{equation}
Consequently \eqref{eq:liminfintro} holds.
\end{corollary}
\begin{proof}
Apply Proposition~\ref{prop:uniform} and the finite Euler product
for $\sigma_{-\alpha}$. For example,
$Z_{qr}=\sigma_{-\alpha}(q)(1+r^{-\alpha})$ when $r\nmid q$.
The primorial product diverges already by comparison with
$\prod_{p\le y}(1+1/p)$. For completeness, the latter is
\[
 \prod_{p\le y}(1-p^{-2})\prod_{p\le y}(1-p^{-1})^{-1}.
\]
The first factor is bounded below by $1/\zeta(2)>0$; the second
is at least $\sum_{n\le y}1/n$, by expanding its finite Euler
product. Finally $Z_m\ge1$ for every $m$, while prime arguments
approach $1$, proving the liminf assertion.
\end{proof}

\subsection{Sharp maximal order}
The primorial argument can be sharpened from divergence to an exact
logarithmic scale. Maximal orders of generalized divisor sums have a
classical history \cite{ramanujan}; the following elementary proof
exploits the useful inequality $2\alpha>1$.

\begin{theorem}[Sharp maximal order]\label{thm:maximal}
Let $p_j$ be the $j$th prime, $P_k=\prod_{j\le k}p_j$, and, for $x\ge2$,
\[
 k(x)=\max\{k\ge1:P_k\le x\}.
\]
Then, as $x\to\infty$,
\begin{equation}\label{eq:maximalexact}
 \log\max_{m\le x}R_m
 =\sum_{j\le k(x)}p_j^{-\alpha}+O(1),
\end{equation}
and consequently
\begin{equation}\label{eq:maximalasymptotic}
 \log\max_{m\le x}R_m
 \sim\frac{(\log x)^{1-\alpha}}
 {(1-\alpha)\log\log x}
 =\ph^2\frac{(\log x)^{2-\ph}}{\log\log x}.
\end{equation}
In particular, $R_m=m^{o(1)}$. The same leading scale holds for the
unnormalized block lengths:
\begin{equation}\label{eq:Bmaximal}
 \log\left(\frac{\max_{m\le x}B(m)}{x^\alpha}\right)
 \sim\frac{(\log x)^{1-\alpha}}
 {(1-\alpha)\log\log x}.
\end{equation}
\end{theorem}
\begin{proof}
If $p^a\Vert m$, then its Euler factor in $Z_m=\sigma_{-\alpha}(m)$
lies between $1+p^{-\alpha}$ and $(1-p^{-\alpha})^{-1}$. Because
$2\alpha>1$, summing the corresponding logarithmic expansions gives,
uniformly in $m$,
\begin{equation}\label{eq:logZsupport}
 \log Z_m=\sum_{p\mid m}p^{-\alpha}+O(1).
\end{equation}
If $m\le x$, then $\omega(m)\le k(x)$, since the product of any $r$
distinct primes is at least $P_r$. As $p^{-\alpha}$ decreases with $p$,
the sum in \eqref{eq:logZsupport} is at most
$\sum_{j\le k(x)}p_j^{-\alpha}$. The matching lower bound, up to $O(1)$,
is attained by $m=P_{k(x)}$. Proposition~\ref{prop:uniform} gives
$\log R_m-\log Z_m=o(1)$ uniformly on every tail, proving
\eqref{eq:maximalexact}.

The prime number theorem \cite[\S27.12]{dlmf} and partial summation give
\[
 \sum_{p\le y}p^{-\alpha}
 \sim\frac{y^{1-\alpha}}{(1-\alpha)\log y},
 \qquad p_{k(x)}\sim\log x.
\]
This proves \eqref{eq:maximalasymptotic} and the subpower bound. For
\eqref{eq:Bmaximal}, use the same upper bound and the primorial lower
bound. By the definition of $k(x)$,
\[
 x<P_{k(x)+1}=P_{k(x)}p_{k(x)+1},\qquad
 0\le\log\frac{x}{P_{k(x)}}<\log p_{k(x)+1}=O(\log\log x).
\]
This loss is negligible on the scale in \eqref{eq:maximalasymptotic}.
\end{proof}

\subsection{A random Euler product}
Recall the prime-adic variables and Euler product from
Theorem~\ref{thm:arithmetic}. For each prime $p$, the variables $V_p$
are independent with distribution
\begin{equation}\label{eq:valuations}
 \Prob(V_p=j)=(1-p^{-1})p^{-j},\qquad j=0,1,2,\ldots.
\end{equation}
Define
\begin{equation}\label{eq:randomZ}
 Z=\prod_p\left(1+p^{-\alpha}+\cdots+p^{-\alpha V_p}\right).
\end{equation}
This is the natural probabilistic model for divisibility of a
uniformly sampled integer. Its use belongs to the classical theory of
arithmetical limiting distributions \cite{erdoswintner}; the elementary
argument below is included to specify precisely what is needed here.

\begin{theorem}[Value-sampled distribution]\label{thm:distribution}
The product \eqref{eq:randomZ} converges to a finite random variable
almost surely. For every bounded continuous $h:[0,\infty)\to\mathbb R$,
\begin{equation}\label{eq:weak}
 \lim_{M\to\infty}\frac1M\sum_{m\le M}h(R_m)=\E[h(Z)].
\end{equation}
Equivalently, the empirical distributions converge weakly to the law
of $Z$. All positive integer moments converge. For an integer $r\ge1$,
\begin{equation}\label{eq:allmoments}
 \lim_{M\to\infty}\frac1M\sum_{m\le M}R_m^r
 =\E[Z^r]
 =\prod_p\left[(1-p^{-1})
   \sum_{j=0}^{\infty}\frac{(1+p^{-\alpha}+\cdots+p^{-j\alpha})^r}{p^j}\right].
\end{equation}
In particular,
\begin{align}
 \E Z&=\zeta(\ph),\label{eq:mean}\\
 \E Z^2&=\frac{\zeta(\ph)^2\zeta(2\ph-1)}{\zeta(2\ph)},
       \label{eq:second}\\
 \Var Z&=\zeta(\ph)^2
 \left(\frac{\zeta(2\ph-1)}{\zeta(2\ph)}-1\right)>0.
       \label{eq:variance}
\end{align}
Here $2\ph-1=\sqrt5$.
\end{theorem}
\begin{proof}
Realize the independent $V_p$ on their countable product probability
space and, for every $d=\prod p^{v_p(d)}$, set
$I_d=\ind_{\{V_p\ge v_p(d)\ \text{for all }p\}}$. Then
$\E I_d=1/d$. Monotone convergence gives
\[
 \E\left[\sum_{d\ge1}d^{-\alpha}I_d\right]
 =\sum_{d\ge1}d^{-1-\alpha}=\zeta(\ph)<\infty.
\]
The nonnegative sum is therefore finite almost surely. Expanding
finite-prime products and passing monotonically to the limit shows
that it equals $Z$ in \eqref{eq:randomZ}.

For a finite cutoff $D$, the arithmetic approximation
\[
 Z_{m,D}=\sum_{\substack{d\le D\\d\mid m}}d^{-\alpha}
\]
is periodic. Uniform residues modulo $\lcm(1,\ldots,D)$, and the
Chinese remainder theorem, identify its limiting distribution with
that of $Z_D=\sum_{d\le D}d^{-\alpha}I_d$. Its arithmetic tail is
uniformly small in mean:
\begin{equation}\label{eq:tailZ}
 \frac1M\sum_{m\le M}(Z_m-Z_{m,D})
 =\sum_{D<d\le M}d^{-\alpha}\frac{\lfloor M/d\rfloor}{M}
 \le\sum_{d>D}d^{-1-\alpha}.
\end{equation}
The probabilistic tail has the same upper bound. First testing
bounded Lipschitz functions and then letting $D\to\infty$ proves
weak convergence for $Z_m$. Also
$M^{-1}\sum_{m\le M}Z_m\le\zeta(\ph)$. Since $R_m/Z_m\to1$, it
follows that
\begin{equation}\label{eq:RZLone}
 \frac1M\sum_{m\le M}|R_m-Z_m|\longrightarrow0.
\end{equation}
To justify the transfer for every bounded continuous test function,
let $J_M$ be uniform on $\{1,\ldots,M\}$. For every $\varepsilon>0$,
Markov's inequality gives
\[
 \Prob(|R_{J_M}-Z_{J_M}|>\varepsilon)
 \le\frac{1}{\varepsilon M}\sum_{m\le M}|R_m-Z_m|
 \longrightarrow0.
\]
The already established convergence $Z_{J_M}\Rightarrow Z$ gives
tightness. Together with this estimate it also gives tightness of
$R_{J_M}$. On a fixed compact interval, a continuous test function is
uniformly continuous, so the difference of its expectations tends to
zero after the tails have been made arbitrarily small. Hence
$R_{J_M}\Rightarrow Z$, without requiring uniform continuity on
the entire half-line.

For integer moments expand the divisor sums. The natural dominating
series is
\begin{equation}\label{eq:lcmseries}
 \sum_{d_1,\ldots,d_r\ge1}
 \frac{(d_1\cdots d_r)^{-\alpha}}{\lcm(d_1,\ldots,d_r)}.
\end{equation}
It equals the Euler product in \eqref{eq:allmoments}. Indeed, for
each prime the factor is the expectation of
$(\sum_{j=0}^{V_p}p^{-j\alpha})^r$. This factor is
$1+O_r(p^{-1-\alpha})$, since $\Prob(V_p\ge1)=p^{-1}$ and
$\sum_{j=0}^{V_p}p^{-j\alpha}=1+O(p^{-\alpha})$ uniformly in
$V_p\ge1$ for large $p$. Hence the product and
\eqref{eq:lcmseries} converge.

The arithmetic $r$th moment is the same nonnegative sum with
$1/\lcm(d_1,\ldots,d_r)$ replaced by
$\lfloor M/\lcm(d_1,\ldots,d_r)\rfloor/M$. Dominated convergence
proves moment convergence for $Z_m$, uniformly bounded in $M$.
Write $R_m=Z_m(1+\varepsilon_m)$, where $\varepsilon_m\to0$.
For $m$ large enough that $|\varepsilon_m|\le1/2$, the mean value theorem
gives
\[
 |R_m^r-Z_m^r|
 \le c_r|\varepsilon_m|Z_m^r,\qquad
 c_r=r(3/2)^{r-1}.
\]
Let $K_r=\sup_{M\ge1}M^{-1}\sum_{m\le M}Z_m^r<\infty$.
For every sufficiently large fixed $D$,
\[
 \frac1M\sum_{m\le M}|R_m^r-Z_m^r|
 \le O_{D,r}(M^{-1})+
 c_rK_r\sup_{m>D}|\varepsilon_m|.
\]
Taking $M\to\infty$ and then $D\to\infty$ transfers the moment limit
to $R_m$.

For $r=2$, write $d=ga$, $e=gb$, $(a,b)=1$. Formula
\eqref{eq:lcmseries} becomes
\[
 \sum_{g\ge1}g^{-2\alpha-1}
 \sum_{(a,b)=1}a^{-\alpha-1}b^{-\alpha-1}
 =\zeta(2\alpha+1)\frac{\zeta(\alpha+1)^2}{\zeta(2\alpha+2)}.
\]
This proves \eqref{eq:second}; \eqref{eq:mean} was already obtained.
The variance is positive since $2\ph-1<2\ph$ and $\zeta(s)$ is
strictly decreasing for real $s>1$.
\end{proof}

\begin{proposition}[Full support]\label{prop:fullsupport}
The topological support of the limiting law is exactly $[1,\infty)$.
In fact, $Z>1$ almost surely, although $1$ remains in the support.
\end{proposition}
\begin{proof}
For each prime put $a_p=\log(1+p^{-\alpha})$. Then $a_p\to0$ and
$\sum_pa_p=\infty$ by the prime number theorem, since $\alpha<1$. The
elementary greedy subsum
argument therefore shows that the finite subsums of the $a_p$ are dense
in $[0,\infty)$. Each such subsum is the logarithm of a finite-prime
product obtained by prescribing $V_p\in\{0,1\}$ at finitely many primes,
an event of positive probability.

For all sufficiently large $p$,
\[
 0\le \log(1+p^{-\alpha}+\cdots+p^{-\alpha V_p})
 \le2p^{-\alpha}\ind_{\{V_p\ge1\}}.
\]
Consequently, the expected logarithmic tail beyond $P$ is
$O(\sum_{p>P}p^{-1-\alpha})=o(1)$. After prescribing any finite pattern,
one may prescribe zeros up to a larger cutoff and use Markov's inequality
on the independent remaining tail. Thus every neighborhood of every
finite product, and hence of every point of $[1,\infty)$, has positive
probability. The reverse inclusion follows from $Z\ge1$. Finally,
$\sum_p\Prob(V_p\ge1)=\sum_p1/p=\infty$, so independence and the second
Borel--Cantelli lemma give $Z>1$ almost surely.
\end{proof}

\begin{theorem}[Singular continuity]\label{thm:singular}
The law of $Z$ is singular continuous: it has no atoms and is
concentrated on a set of Lebesgue measure zero.
\end{theorem}
\begin{proof}
Set
\[
 X_p=\log\left(\sum_{j=0}^{V_p}p^{-j\alpha}\right),\qquad
 Y=\log Z=\sum_pX_p.
\]
By Theorem~\ref{thm:distribution}, this sum of independent discrete
random variables converges almost surely. The Jessen--Wintner
pure-type theorem \cite[Theorem~35]{jessenwintner} therefore makes
its law either discrete, singular continuous, or absolutely continuous;
the additive-function formulation is also stated in
\cite[\S13(iv), p.~720]{erdoswintner}.

The values of $X_p$ strictly increase with $V_p$. Its largest atom is
therefore the one at zero, of mass
\[
 d_p=\sup_{x\in\mathbb R}\Prob(X_p=x)=\Prob(V_p=0)=1-\frac1p.
\]
L\'evy's continuity criterion for a convergent infinite convolution of
discrete laws \cite[Th\'eor\`eme~XIII]{levy}, as reproduced in
\cite[\S13(v), p.~721]{erdoswintner}, excludes atoms precisely when
the product of their largest atom masses is zero. Here
\[
 \prod_p d_p=0,\qquad\text{since}\qquad
 \sum_p(1-d_p)=\sum_p\frac1p=\infty.
\]
Thus the law of $Y$ has no atoms.

To exclude absolute continuity, we adapt the covering argument of
Erd\H{o}s \cite[\S3]{erdossmoothness}, retaining the full prime-power
model. Put $f(n)=\log\sigma_{-\alpha}(n)$ and, for $y\ge2$, let
\[
 M_y=\prod_{p\le y}p^{V_p},\qquad
 Y_y=\log\sigma_{-\alpha}(M_y),\qquad
 c_y=\prod_{p\le y}(1-p^{-1}).
\]
For every integer $m$ whose prime factors are at most $y$,
independence gives $\Prob(M_y=m)=c_y/m$ exactly. Every
$m\le y^{\alpha/4}$ has this property. Mertens' product formula
\cite{mertens} and the harmonic-sum asymptotic therefore give
\[
 \Prob(M_y\le y^{\alpha/4})
 =c_y\sum_{m\le y^{\alpha/4}}\frac1m
 \longrightarrow e^{-\gamma_E}\frac{\alpha}{4}>0,
\]
where $\gamma_E$ is Euler's constant. On the other hand,
\[
 0\le\E(Y-Y_y)\ll\sum_{p>y}p^{-1-\alpha}\ll y^{-\alpha}.
\]
Markov's inequality shows that $Y-Y_y\le y^{-\alpha/2}$ except on
an event of probability $O(y^{-\alpha/2})$. Consequently the set
\[
 H_y=\bigcup_{m\le y^{\alpha/4}}
 [\,f(m),\,f(m)+y^{-\alpha/2}\,]
\]
has Lebesgue measure at most $y^{-\alpha/4}\to0$, whereas
$\liminf_{y\to\infty}\Prob(Y\in H_y)>0$. This contradicts absolute
continuity, whose integrable density would give uniformly vanishing
mass to sets of vanishing Lebesgue measure. The pure-type theorem
therefore makes the law of $Y$ singular continuous.
Exponentiation preserves atoms and maps null sets to null sets on
each bounded interval, so the same conclusion holds for $Z$.
\end{proof}

Singularity and full support are compatible: every nonempty open
subinterval of $[1,\infty)$ has positive probability, but almost all
the probability is carried by a Lebesgue-null subset. The classical
pure-type theorem is used here only to establish this refinement;
the preceding convergence and moment arguments are direct.

\begin{proposition}[Real Mellin moments]\label{prop:mellin}
For every real $t$, the quantity
\begin{equation}\label{eq:mellinproduct}
 \mathcal M(t)=\E Z^t
 =\prod_p\left[(1-p^{-1})\sum_{j=0}^{\infty}
 \frac{(1+p^{-\alpha}+\cdots+p^{-j\alpha})^t}{p^j}\right]
\end{equation}
is finite, and
\begin{equation}\label{eq:realmellin}
 \lim_{M\to\infty}\frac1M\sum_{m\le M}R_m^t=\mathcal M(t).
\end{equation}
\end{proposition}
\begin{proof}
If $t\ge0$, choose an integer $r>t$ and use $1\le Z^t\le1+Z^r$;
Theorem~\ref{thm:distribution} makes the right side integrable. If $t<0$,
then $0<Z^t\le1$. Finite-prime products converge almost surely to $Z$,
and monotone convergence for $t\ge0$, or dominated convergence for $t<0$,
gives the factorization \eqref{eq:mellinproduct}.

For positive $t$, weak convergence together with the uniformly bounded
higher integer moments gives uniform integrability and
\eqref{eq:realmellin}. For negative $t$, use $Z_m\ge1$ and
$R_m/Z_m\to1$: eventually $R_m\ge1/2$, so the test function is bounded on
the relevant range. Replacing $x^t$ below $1/2$ by a bounded continuous
extension makes the weak-convergence argument literal without changing
the tail of the arithmetic sequence. The case $t=0$ is immediate.
\end{proof}

\section{Sampling through the self-generated clock}\label{sec:sampling}
\subsection{The size-biased limiting law}
Sampling a block label uniformly is different from sampling a sequence
position uniformly. The latter visits a block in proportion to its
length, so it preserves more of the large-divisor fluctuations.

\begin{theorem}[Position-sampled distribution]\label{thm:sizebias}
Let $U_N$ be uniform on $\{1,\ldots,N\}$ and set
\[
 X_N=\frac{B(A(U_N))}{C A(U_N)^\alpha}.
\]
Then $X_N$ converges in distribution to the size-biased version $Z^*$
of $Z$, characterized by
\begin{equation}\label{eq:sizebias}
 \E[h(Z^*)]=\frac{\E[Z h(Z)]}{\zeta(\ph)}
\end{equation}
for bounded continuous $h$. For every real $t$, the Mellin moments also
converge:
\begin{equation}\label{eq:sizebiasallmoments}
 \lim_{N\to\infty}\E X_N^t
 =\E (Z^*)^t
 =\frac{\mathcal M(t+1)}{\mathcal M(1)}
 =\frac{\mathcal M(t+1)}{\zeta(\ph)}.
\end{equation}
In particular, its first moment is
\begin{equation}\label{eq:sizebiasedmean}
 \E Z^*=\frac{\zeta(\ph)\zeta(2\ph-1)}{\zeta(2\ph)}
 =2.830405204114929048\ldots.
\end{equation}
\end{theorem}
\begin{proof}
First take $N=F(M)$. Exact block counting gives
\begin{equation}\label{eq:sizebiasexact}
 \E[h(X_{F(M)})]
 =\frac1{F(M)}\sum_{m\le M}B(m)h(R_m)
 =\frac C{F(M)}\sum_{m\le M}m^\alpha R_m h(R_m).
\end{equation}
The weak convergence and bounded second moments from
Theorem~\ref{thm:distribution} imply uniform integrability of the
first moments. Explicitly, for $T>0$,
\[
 \frac1M\sum_{\substack{m\le M\\R_m>T}}|R_mh(R_m)|
 \le\frac{\|h\|_\infty}{TM}\sum_{m\le M}R_m^2=O_h(T^{-1}),
\]
uniformly in $M$. Truncating first and then letting $T\to\infty$ gives
\[
 \frac1M\sum_{m\le M}R_m h(R_m)\longrightarrow\E[Zh(Z)].
\]
If a sequence $g_m$ has Ces\`aro mean $L$, partial summation gives
\begin{equation}\label{eq:weightedcesaro}
 \sum_{m\le M}m^\alpha g_m
 =\frac{L}{\ph}M^\ph+o(M^\ph).
\end{equation}
Apply this to the last display and use
$F(M)\sim C\zeta(\ph)M^\ph/\ph$. This proves
\eqref{eq:sizebias} along endpoints.

For sufficiently large $N$ put $M=A(N)$ and $r=N-F(M-1)$, so
$1\le r\le B(M)$. The exact decomposition
\begin{equation}\label{eq:partialblockmixture}
 \E h(X_N)=\frac{F(M-1)}N\E h(X_{F(M-1)})
             +\frac rN h(R_M)
\end{equation}
shows directly that bounded test functions have the same limit: indeed,
$B(M)\ll M=o(F(M))$.

For a real $t$, apply \eqref{eq:weightedcesaro} at complete endpoints
with $g_m=R_m^{t+1}$ and use Proposition~\ref{prop:mellin}. This proves
\eqref{eq:sizebiasallmoments} along $N=F(M)$. In
\eqref{eq:partialblockmixture}, now with $h(x)=x^t$, the final term is
harmless for $t<0$ because eventually $R_M\ge1/2$. For $t\ge0$, its
absolute size is at most
\[
 \frac{B(M)R_M^t}{N}
 \ll\frac{R_M^{t+1}}M=M^{-1+o(1)}
\]
by Theorem~\ref{thm:maximal}. The changed normalization is accounted for
by the first coefficient in \eqref{eq:partialblockmixture}, which tends to
$1$. This proves all real moment limits. Formula
\eqref{eq:sizebiasedmean} follows from \eqref{eq:second} divided by
\eqref{eq:mean}.
\end{proof}

This gives two distinct, explicit constants for the same local
quantity, depending on whether one samples values or positions.
The global time change is generated by the sequence itself; the
limiting local law remains arithmetic. Moreover, $Z^*$ is also singular
continuous with support $[1,\infty)$: by \eqref{eq:sizebias}, its law has
the strictly positive density $z/\zeta(\ph)$ relative to the law of $Z$,
so the two measures have the same null sets and the same support.

\subsection{Residue classes and prime values at sequence positions}
The distribution of block labels seen from positions also has exact
arithmetic densities. These are position-sampled densities, not densities
under uniform sampling of the block label $m$.

\begin{theorem}[Residue-class densities]\label{thm:residues}
Fix $q\ge2$ and a residue class $a\pmod q$. Then
\begin{equation}\label{eq:residuedensitydef}
 \delta_{q,a}:=\lim_{N\to\infty}\frac1N
 \#\{n\le N:A(n)\equiv a\pmod q\}
\end{equation}
exists and equals
\begin{equation}\label{eq:residuedensity}
 \delta_{q,a}=\frac1q\sum_{g\mid\gcd(a,q)}g^{1-\ph}
 \prod_{p\mid q/g}(1-p^{-\ph}),
\end{equation}
where $\gcd(0,q)=q$. In particular, if $\phi_{\!E}$ denotes Euler's
totient function, then
\begin{equation}\label{eq:divisibilitydensity}
 \operatorname{dens}\{n:q\mid A(n)\}
 =\delta_{q,0}=\frac1q\sum_{d\mid q}\frac{\phi_{\!E}(d)}{d^\ph}.
\end{equation}
The function $q\mapsto\delta_{q,0}$ is multiplicative. Consequently,
divisibility of $A(n)$ by fixed coprime moduli is asymptotically
independent. For example,
\begin{align}
 \operatorname{dens}\{n:A(n)\text{ is even}\}
 &=\frac{1+2^{-\ph}}2=0.6628895560\ldots,\label{eq:evendensity}\\
 \operatorname{dens}\{n:A(n)\text{ is odd}\}
 &=\frac{1-2^{-\ph}}2=0.3371104439\ldots.\label{eq:odddensity}
\end{align}
Every reduced residue class modulo $q$ has density
$q^{-1}\prod_{p\mid q}(1-p^{-\ph})$.
\end{theorem}
\begin{proof}
At a complete endpoint, the numerator in \eqref{eq:residuedensitydef} is
\[
 \sum_{\substack{m\le M\\m\equiv a\ (\mathrm{mod}\ q)}}B(m).
\]
Proposition~\ref{prop:uniform} permits $B(m)$ to be replaced uniformly by
$C\sigma_\alpha(m)$ with total error $o(M^\ph)$. Write $m=ek$ in the
divisor sum. If $g=(e,q)$, the congruence $ek\equiv a\pmod q$ is soluble
exactly when $g\mid a$; in that case $k$ occupies one residue class modulo
$q/g$. Uniformly for the finitely many possible moduli,
\[
 \sum_{\substack{k\le x\\ek\equiv a\ (\mathrm{mod}\ q)}}k^\alpha
 =\frac{g}{q\ph}x^\ph+O_q(x^\alpha+1).
\]
It follows that
\begin{equation}\label{eq:sigmainprogression}
 \sum_{\substack{m\le M\\m\equiv a\ (\mathrm{mod}\ q)}}
 \sigma_\alpha(m)
 =\frac{M^\ph}{\ph q}
 \sum_{\substack{e\ge1\\(e,q)\mid a}}\frac{(e,q)}{e^\ph}
 +O_q(M).
\end{equation}
Indeed, the summed progression errors are
$O_q(M^\alpha\sum_{e\le M}e^{-\alpha}+M)=O_q(M)$, and extending the
main sum past $M$ also costs $O_q(M)$. Division by
$F(M)\sim C\zeta(\ph)M^\ph/\ph$ gives the first expression
\[
 \delta_{q,a}=\frac1{q\zeta(\ph)}
 \sum_{\substack{e\ge1\\(e,q)\mid a}}\frac{(e,q)}{e^\ph}.
\]
Partitioning by $g=(e,q)$ and using
$\sum_{(r,q/g)=1}r^{-\ph}=\zeta(\ph)
\prod_{p\mid q/g}(1-p^{-\ph})$ proves \eqref{eq:residuedensity}.
For arbitrary $N$, the incomplete block has size
$B(A(N))=O(A(N))=o(N)$, so it does not change the limit.

The identity $(e,q)=\sum_{d\mid(e,q)}\phi_{\!E}(d)$ gives
\eqref{eq:divisibilitydensity}; multiplicativity and the remaining
specializations follow at once.
\end{proof}

The prime-number theorem produces a complementary zero-density result
with an explicit first-order count.

\begin{corollary}[Positions carrying prime values]\label{cor:primepositions}
Let
\[
 \mathcal P_A(N)=\#\{n\le N:A(n)\text{ is prime}\}.
\]
Then
\begin{equation}\label{eq:primepositions}
 \mathcal P_A(N)\sim\frac{\ph}{\zeta(\ph)}\frac{N}{\log N}
 =0.7228741287\ldots\frac{N}{\log N}.
\end{equation}
\end{corollary}
\begin{proof}
Proposition~\ref{prop:primecode} and exact block counting give
\[
 \mathcal P_A(F(M))=\sum_{p\le M}B(p)
 =\sum_{p\le M}(A(p)+1).
\]
The prime number theorem \cite[\S27.12]{dlmf}, partial summation, and
$A(p)\sim Cp^\alpha$
yield
\[
 \mathcal P_A(F(M))\sim\frac{CM^\ph}{\ph\log M}.
\]
Combine this with $F(M)\sim C\zeta(\ph)M^\ph/\ph$ and
$\log F(M)\sim\ph\log M$. Away from endpoints, the possible change is at
most $B(A(N))=O(N^\alpha)=o(N/\log N)$, proving the result.
\end{proof}

\section{Arithmetic dependence between blocks}\label{sec:correlations}
The next result shows that consecutive normalized block lengths are not
asymptotically independent.

\begin{theorem}[Fixed-shift correlations]\label{thm:correlations}
For every fixed integer $h\ge1$,
\begin{equation}\label{eq:fixedshift}
 \lim_{M\to\infty}\frac1M\sum_{m\le M}R_mR_{m+h}
 =\mathcal C(h):=
 \sum_{\substack{d,e\ge1\\(d,e)\mid h}}
 \frac{d^{-\alpha}e^{-\alpha}}{\lcm(d,e)}.
\end{equation}
The series converges absolutely, and for every $h\ge1$ it simplifies to
\begin{equation}\label{eq:closedshift}
 \mathcal C(h)=\frac{\zeta(\ph)^2}{\zeta(2\ph)}
              \sigma_{-(2\alpha+1)}(h).
\end{equation}
In particular,
\begin{equation}\label{eq:adjacentcrossmoment}
 \mathcal C(1)=\frac{\zeta(\ph)^2}{\zeta(2\ph)}.
\end{equation}
Thus adjacent normalized block lengths have limiting covariance
\begin{equation}\label{eq:adjacentcovariance}
 \zeta(\ph)^2\left(\frac1{\zeta(2\ph)}-1\right)<0
\end{equation}
and limiting correlation coefficient
\begin{equation}\label{eq:adjacentcorrelation}
 \frac{1-\zeta(2\ph)}{\zeta(\sqrt5)-\zeta(2\ph)}
 =-0.5249104386\ldots.
\end{equation}
\end{theorem}
\begin{proof}
Put
$Z_{m,D}=\sum_{d\le D,\,d\mid m}d^{-\alpha}$. For fixed $D$, the Chinese
remainder theorem gives
\[
 \lim_{M\to\infty}\frac1M\sum_{m\le M}Z_{m,D}Z_{m+h,D}
 =\sum_{\substack{d,e\le D\\(d,e)\mid h}}
 \frac{d^{-\alpha}e^{-\alpha}}{\lcm(d,e)}.
\]
Indeed, the congruences $m\equiv0\pmod d$ and
$m\equiv-h\pmod e$ are soluble exactly when $(d,e)\mid h$, and then
form one residue class modulo $\lcm(d,e)$. The series in
\eqref{eq:fixedshift} is a subseries of the positive, convergent
second-moment series \eqref{eq:lcmseries}. The same series shows that
$Z_m-Z_{m,D}$ tends to zero in Ces\`aro $L^2$ in the precise sense
\[
 \lim_{D\to\infty}\limsup_{M\to\infty}\frac1M
 \sum_{m\le M}(Z_m-Z_{m,D})^2=0.
\]
More explicitly, put
\[
 S_D=\sum_{d,e>D}\frac{d^{-\alpha}e^{-\alpha}}{\lcm(d,e)}
 \longrightarrow0.
\]
The unshifted mean square tail is at most $S_D$ for every $M$, and
the shifted one is at most $(M+h)S_D/M$, by enlarging the summation
interval to $1\le m\le M+h$. Cauchy--Schwarz therefore permits passage
to the untruncated $Z_m$.
For the final transfer write $R_m=Z_m(1+\varepsilon_m)$, with
$\varepsilon_m\to0$. Then
\[
 |R_mR_{m+h}-Z_mZ_{m+h}|
 \le\bigl(|\varepsilon_m|+|\varepsilon_{m+h}|
       +|\varepsilon_m\varepsilon_{m+h}|\bigr)Z_mZ_{m+h}.
\]
If $K_2=\sup_{L\ge1}L^{-1}\sum_{m\le L}Z_m^2<\infty$, Cauchy--Schwarz
gives
\[
 \frac1M\sum_{m\le M}Z_mZ_{m+h}
 \le K_2\sqrt{1+h/M}.
\]
Put $e_D=\sup_{m>D}|\varepsilon_m|$. Splitting off the first $D$ terms,
the averaged product difference is at most
$O_{D,h}(M^{-1})+(2e_D+e_D^2)K_2\sqrt{1+h/M}$.
Taking $M\to\infty$ and then $D\to\infty$ proves the transfer to $R_m$.

For arbitrary $h$, write $d=ga$, $e=gb$, with $(a,b)=1$.
The compatibility condition becomes $g\mid h$, and
$\lcm(d,e)=gab$. Thus
\[
 \mathcal C(h)=
 \sum_{g\mid h}g^{-2\alpha-1}\sum_{(a,b)=1}(ab)^{-\ph}.
\]
The inner sum has Euler product
\[
 \sum_{(d,e)=1}(de)^{-\ph}
 =\prod_p\frac{1+p^{-\ph}}{1-p^{-\ph}}
 =\frac{\zeta(\ph)^2}{\zeta(2\ph)}.
\]
This proves \eqref{eq:closedshift} and its specialization at $h=1$.
For fixed $h$, the shifted first and second Ces\`aro moments have the
same limits as the unshifted ones: the finitely many boundary terms are
$M^{o(1)}$ by Theorem~\ref{thm:maximal}, and their contribution after
division by $M$ vanishes.
Subtracting the square of the mean in \eqref{eq:mean}, then dividing by
the variance in \eqref{eq:variance}, proves
\eqref{eq:adjacentcovariance}--\eqref{eq:adjacentcorrelation}.
\end{proof}

\section{Conclusion and further directions}\label{sec:conclusion}

Divisor weighting separates two aspects of self-description.
At the global level, inversion contracts the asymptotic envelopes and
forces a single power law. The exponent is stable throughout the
finite-Dirichlet-mass class; the kernel changes the coefficient.
At the local level, divisor structure persists. The resulting
multiplicities have unbounded normalized extremes, a singular
continuous limiting law with full support, and explicit arithmetic
dependence at fixed shifts.

The uniform divisor comparison is the bridge between these levels.
It permits classical arithmetic statistics to be transferred to a
sequence that is not itself multiplicative. The self-generated clock
then adds a second transformation: a position sees the size-biased
law, not the law seen by a uniformly chosen block label. The exact
residue densities show this distinction concretely; even values, for
example, occupy more than half of the sequence positions.

Several questions now have a precise starting point.

\paragraph{The endpoint error and an oscillatory profile.}
Can the estimate of Theorem~\ref{thm:rate} be sharpened to
\begin{equation}\label{eq:openrate}
 A(n)-Cn^\alpha=O\left(\frac{n^\alpha}{\log n}\right)?
\end{equation}
For the classical Golomb sequence this scale is established
\cite{vardi}, and P\'etermann--R\'emy \cite{petermannremy} obtain a
continuous antiperiodic profile on the $\log\log n$ scale.
A natural divisor-weighted analogue would ask for a nonzero continuous
function $P$, with $P(t+1)=-P(t)$, such that
\[
 A(n)=Cn^\alpha\left[
 1+\frac{1}{\log n}P\left(\frac{\log\log n}{\log\ph}\right)
 +o\left(\frac1{\log n}\right)\right].
\]
Neither this expansion nor \eqref{eq:openrate} is proved here.
The formal model in Appendix~\ref{app:formal} explains why this
particular shape is plausible; the sign changes in
Table~\ref{tab:growth} are illustrations, not a proof of persistent
oscillation.

\paragraph{The boundary of the kernel theorem.}
What replaces the power law when $W(\ph)=\infty$?
The present proof identifies the exact step that fails: the clock's
comparison sum no longer has a finite $m^\ph$ coefficient.
Kernels at or near this boundary suggest logarithmic corrections,
whereas heavier kernels may alter the exponent. A useful next goal is
a classification in terms of the growth of
$\sum_{q\le x}w(q)q^{-\ph}$, with causality kept separate from
asymptotic analysis.

\paragraph{Quantitative arithmetic statistics.}
The limiting law has full support but no density. What bounds can be
obtained for its concentration function
\[
 \sup_{x\in\mathbb R}\Prob(x<Z\le x+\varepsilon)
 \qquad(\varepsilon\downarrow0),
\]
and how effectively do the empirical block distributions approach it?
The uniform estimate \eqref{eq:uniform} gives one source of error;
approximating divisibility by finitely many primes gives another.
Combining them quantitatively, and extending fixed-shift correlations
to joint laws of several consecutive blocks, would refine the present
description of local fluctuations.

\paragraph{Beyond the real boundary of the Dirichlet series.}
Appendix~\ref{app:series} establishes an abscissa of convergence and a
real-axis boundary asymptotic for $D_A(s)$. Whether this series admits
continuation across any part of $\Rea s=\ph$ is not settled by those
results. A sharper error expansion may provide the first indication
of what singularities should be expected.

These directions concern the second-order behavior and the interaction
between arithmetic and inversion. The first-order law, its kernel
dependence, and the two sampling limits do not rely on resolving them.

\appendix
\section{Additional structure of the self-description}\label{app:structure}
\subsection{A causal nested recurrence and arithmetic inversion}
\begin{proposition}[A nested recurrence]\label{prop:nested}
The same sequence is generated by $A(1)=1$ and
\begin{equation}\label{eq:nested}
 A(n+1)=A(n)+
 \ind_{\left\{\displaystyle n=
  \sum_{d=1}^{A(n)}A(d)\left\lfloor\frac{A(n)}d\right\rfloor\right\}}
 \qquad(n\ge1).
\end{equation}
Every sequence argument on the right is at most $n$.
\end{proposition}
\begin{proof}
At position $n$ the value increases precisely when the current block
ends, that is, when $n=F(A(n))$. Apply \eqref{eq:floor}. Since the
sequence starts at $1$ and has increments in $\{0,1\}$, $A(n)\le n$;
all terms needed on the right are already available. Conversely,
the block construction satisfies this explicit recurrence, whose
successive terms are unique.
\end{proof}

There is also exact arithmetic recovery of the sequence from its
multiplicities:
\begin{equation}\label{eq:mobius}
 A(m)=\sum_{d\mid m}\mu(d)B(m/d).
\end{equation}
This is ordinary M\"obius inversion \cite{dlmf}. Neither $A$ nor $B$
is multiplicative: $A(6)=3\ne A(2)A(3)=4$ and
$B(6)=8\ne B(2)B(3)=9$.

\subsection{Fixed iterates and the time to reach the fixed point}
Write $A^{\circ0}(n)=n$ and
$A^{\circ(k+1)}(n)=A(A^{\circ k}(n))$.

\begin{proposition}[Iterates and hitting time]\label{prop:iterates}
For every fixed integer $k\ge1$ and every $0<\gamma<1$,
\begin{equation}\label{eq:iterateasymptotic}
 A^{\circ k}(n)=
 C^{(1-\alpha^k)/(1-\alpha)}n^{\alpha^k}
 \left(1+O_{k,\gamma}((\log n)^{-\gamma})\right).
\end{equation}
The only fixed points of $A$ are $1$ and $2$. For $n\ge2$, let
\begin{equation}\label{eq:hittingdef}
 h(n)=\min\{k\ge0:A^{\circ k}(n)=2\}.
\end{equation}
Then
\begin{equation}\label{eq:hittingtime}
 h(n)=\frac{\log\log n}{\log\ph}+O(1).
\end{equation}
\end{proposition}
\begin{proof}
Formula \eqref{eq:iterateasymptotic} follows by induction from
Theorem~\ref{thm:rate}. For fixed $j$,
$\log A^{\circ j}(n)\asymp\log n$, so the relative error remains
$O_{j,\gamma}((\log n)^{-\gamma})$, while the exponent of $C$ obeys
$s_{j+1}=1+\alpha s_j$ and hence
$s_j=1+\alpha+\cdots+\alpha^{j-1}$.

For the hitting time, the power envelope gives a constant $M\ge2$ such
that
\[
 M^{-1}x^\alpha\le A(x)\le Mx^\alpha\qquad(x\ge1).
\]
If $x_j=A^{\circ j}(n)$, iteration of the logarithmic inequalities gives
\begin{equation}\label{eq:iteratelogenvelope}
 \left|\log x_j-\alpha^j\log n\right|
 \le\frac{\log M}{1-\alpha}.
\end{equation}
Since $A(x)<x$ for $x\ge3$ and $A(x)\ge2$ for $x\ge2$, every orbit
starting at $n\ge2$ reaches $2$, and no integer greater than $2$ is fixed.
Putting $j=h(n)$ in \eqref{eq:iteratelogenvelope} shows that
$\alpha^j\log n=O(1)$, which gives the lower bound in
\eqref{eq:hittingtime}. Conversely, after
$\log\log n/\log\ph+O(1)$ iterations the upper half of
\eqref{eq:iteratelogenvelope} places $x_j$ in a fixed finite set. Strict
descent then requires only a bounded number of further steps. This proves
the matching upper bound.
\end{proof}

\subsection{A zeta law inside the endpoint clock}
The divisor decomposition of a block carries a natural latent label. For
each $m$, partition its $B(m)$ copies into groups indexed by $d\mid m$, the
$d$th group having $A(d)$ copies, and mark that group by the complementary
divisor $q=m/d$. Choose one marked copy uniformly from all blocks through
$M$, and denote its mark by $Q_M$.

\begin{theorem}[Divisor-layer zeta law]\label{thm:zetalayer}
The variables $Q_M$ converge in total variation to the zeta-distributed
integer $Q$ given by
\begin{equation}\label{eq:zetalayerlaw}
 \Prob(Q=q)=\frac{q^{-\ph}}{\zeta(\ph)}\qquad(q\ge1).
\end{equation}
Its tail and moments satisfy
\begin{align}
 \Prob(Q>y)&\sim\frac{y^{-\alpha}}{\alpha\zeta(\ph)},
       \label{eq:zetalayertail}\\
 \E Q^t&=\frac{\zeta(\ph-t)}{\zeta(\ph)}\qquad(t<\alpha).
       \label{eq:zetalayermoments}
\end{align}
The moments of order $t\ge\alpha$, including the mean, are infinite.
For every $t<\alpha$, the corresponding moments of $Q_M$ converge to
\eqref{eq:zetalayermoments}.
\end{theorem}
\begin{proof}
Exact counting and \eqref{eq:floor} give
\begin{equation}\label{eq:zetalayerfinite}
 \Prob(Q_M=q)=\frac{\sum_{d\le M/q}A(d)}{F(M)}.
\end{equation}
Since
\[
 \sum_{d\le x}A(d)\sim\frac C\ph x^\ph,\qquad
 F(M)\sim\frac{C\zeta(\ph)}\ph M^\ph,
\]
the right side tends to $q^{-\ph}/\zeta(\ph)$ for every fixed $q$.
Both sides of the limiting assertion are probability mass functions, so
pointwise convergence implies total-variation convergence. The power
envelope also gives the uniform bound
$\Prob(Q_M=q)\ll q^{-\ph}$. Dominated convergence therefore proves moment
convergence for $t<\ph-1=\alpha$. Finally, summing the power-law mass in
\eqref{eq:zetalayerlaw} proves \eqref{eq:zetalayertail} and the remaining
moment assertions.
\end{proof}

Thus the factor $\zeta(\ph)$ in the global clock has a direct
probabilistic interpretation: a typical divisor contribution has a
complementary divisor with mass exponent $\ph$ and tail exponent $\alpha$.

\subsection{A formal model for second-order oscillations}\label{app:formal}
This subsection records a heuristic only; none of the proved results
depends on it. The antiperiodic shape is suggested by the established
classical Golomb theory \cite{petermannremy}. To see how divisor
weighting could enter the corresponding mechanism, introduce a smooth \emph{formal}
perturbation
\[
 A(e^t)\approx Ce^{\alpha t}(1+u(t)).
\]
The continuum approximation to the weighted clock averages
$u(t-Y_0)$, where
\begin{equation}\label{eq:Yformal}
 Y_0=\log Q+T,\qquad
 \Prob(Q=q)=\frac{q^{-\ph}}{\zeta(\ph)},\qquad
 T\ \text{has density }\ph e^{-\ph v}\ind_{v\ge0},
\end{equation}
with $Q$ and $T$ independent. This follows by replacing
$\sum_{d\le e^t/q}d^\alpha u(\log d)$ by its integral and putting
$d=e^t e^{-v}/q$. Linearizing the inverse then suggests
\begin{equation}\label{eq:formalfeedback}
 u(t)\approx-\alpha\,\E\bigl[u(\alpha t+\log C-Y_0)\bigr].
\end{equation}
The normalization $C(CK)^\alpha=1$ is responsible for the shift
$\log C$. If the order-one shifts are ignored, the reduced
relation $u(t)\approx-\alpha u(\alpha t)$ admits profiles
\begin{equation}\label{eq:formalprofile}
 u(t)=t^{-1}P\left(\frac{\log t}{\log\ph}\right),
 \qquad P(x+1)=-P(x).
\end{equation}
Thus an alternating oscillation on a logarithm-of-logarithm scale
is compatible with a relative error of order $1/\log n$.
Equations \eqref{eq:formalfeedback}--\eqref{eq:formalprofile} are a
heuristic linearization, not an asymptotic theorem. Controlling the
discrete remainder and determining whether a nonzero profile survives
are genuine additional tasks.

\section{Generating series and their real boundary}\label{app:series}
For $|z|<1$ define
\[
 \mathcal A(z)=\sum_{n\ge1}A(n)z^n,
 \qquad \mathcal B(z)=\sum_{m\ge1}B(m)z^m.
\]
Polynomial bounds justify rearrangement, giving the exact Lambert
identity
\begin{equation}\label{eq:lambert}
 \mathcal B(z)=\sum_{d\ge1}\frac{A(d)z^d}{1-z^d}.
\end{equation}
The first differences of $A$, extended by $A(0)=0$, equal $1$
exactly at the block starts. Therefore
\begin{equation}\label{eq:jumpgf}
 (1-z)\mathcal A(z)=\sum_{m\ge1}z^{F(m-1)+1}.
\end{equation}
Together, these formulas link an arithmetic divisor transform to a
sparse sequence of self-generated jump positions.

For $\Rea s>\ph$ put
\[
 D_A(s)=\sum_{n\ge1}\frac{A(n)}{n^s},\qquad
 D_B(s)=\sum_{m\ge1}\frac{B(m)}{m^s}.
\]
Absolute convergence and \eqref{eq:frequency} imply
\begin{equation}\label{eq:Dirichlet}
 D_B(s)=\zeta(s)D_A(s)\qquad(\Rea s>\ph).
\end{equation}
The abscissa of absolute convergence of both series is $\ph$:
$A(n)\sim Cn^\alpha$ gives this for $D_A$, and
$B(n)\ge A(n)$ together with \eqref{eq:Dirichlet} gives it for
$D_B$.

\begin{proposition}[Real boundary asymptotics]\label{prop:abelian}
For every $0<\gamma<1$, as $\varepsilon\downarrow0$,
\begin{align}
 \varepsilon D_A(\ph+\varepsilon)
 &=C+O_\gamma(\varepsilon^\gamma),\label{eq:abelA}\\
 \varepsilon D_B(\ph+\varepsilon)
 &=C\zeta(\ph)+O_\gamma(\varepsilon^\gamma).
       \label{eq:abelB}
\end{align}
\end{proposition}
\begin{proof}
Theorem~\ref{thm:rate} gives
\[
 D_A(\ph+\varepsilon)
 =C\zeta(1+\varepsilon)
 +O_\gamma\left(1+\sum_{n\ge3}
 \frac{1}{n^{1+\varepsilon}(\log n)^\gamma}\right).
\]
The sum is $O_\gamma(\varepsilon^{\gamma-1})$, by the integral
substitution $t=\log x$. Since
$\varepsilon\zeta(1+\varepsilon)=1+O(\varepsilon)$, this proves
\eqref{eq:abelA}. Multiplication by
$\zeta(\ph+\varepsilon)=\zeta(\ph)+O(\varepsilon)$ proves
\eqref{eq:abelB}.
\end{proof}

The real-axis limits above do not imply meromorphic continuation or a
simple pole in a complex neighborhood of $\ph$. The identity
\eqref{eq:Dirichlet} is established in the common half-plane of
absolute convergence; $D_A$ has no Euler product supplied by this
identity, since its coefficients are not multiplicative.

\section{Computation and reproducibility}\label{app:verification}
The computation uses block endpoints, not a materialized sequence of
hundreds of billions of terms. A self-feeding divisor sieve generates
$A(m)$, $B(m)$, and $F(m)$ for $1\le m\le M$. Once $A(m)$ is known,
its contribution is added to every future multiple of $m$. A monotone
pointer in the completed endpoint array determines the next required
sequence value. This costs $O(M\log M)$ integer additions and $O(M)$
memory. The causal initialization is covered by
Proposition~\ref{prop:construction}.

The accompanying C++ implementation's recorded large run contains $M=10^7$ complete blocks,
with
\begin{equation}\label{eq:largest}
 F(10^7)=260\,033\,782\,794.
\end{equation}
Thus it represents the sequence through this position in compressed
block form. In particular,
\[
 A(10^7)=18\,810,\qquad A(10^{11})=5\,541\,027.
\]
Frequency and endpoint arithmetic used unsigned 64-bit integers with
explicit overflow checks. Decimal ratios were computed separately;
they do not enter the integer recurrence.

\begin{table}[ht]
\centering
\caption{Selected position values and comparison with the proved
leading term. The relative error is not monotone.}
\label{tab:growth}
\begin{tabular}{rrrr}
\toprule
$n$&$A(n)$&$A(n)/n^\alpha$&$A(n)/(Cn^\alpha)-1$\\
\midrule
$10^3$&62&0.867540430&$-0.017973995$\\
$10^4$&260&0.876672519&$-0.007636783$\\
$10^5$&1\,088&0.884013611&$+0.000673081$\\
$10^6$&4\,532&0.887331642&$+0.004428978$\\
$10^7$&18\,810&0.887464264&$+0.004579102$\\
$10^8$&77\,920&0.885884359&$+0.002790703$\\
$10^9$&322\,650&0.883946073&$+0.000596630$\\
$10^{10}$&1\,336\,601&0.882392675&$-0.001161763$\\
$10^{11}$&5\,541\,027&0.881487971&$-0.002185857$\\
\midrule
Limit&&0.883418998&0\\
\bottomrule
\end{tabular}
\end{table}

\begin{table}[ht]
\centering
\caption{Empirical moments of $R_m=B(m)/(Cm^\alpha)$ under uniform
sampling of block labels. These finite values need not approach their
limits monotonically.}
\label{tab:moments}
\begin{tabular}{rrr}
\toprule
$M$&$M^{-1}\sum_{m\le M}R_m$&$M^{-1}\sum_{m\le M}R_m^2$\\
\midrule
$10^3$&2.242850506&6.172339582\\
$10^4$&2.222963938&6.181730594\\
$10^5$&2.228275056&6.248626174\\
$10^6$&2.239666494&6.327784209\\
$10^7$&2.246219532&6.373906040\\
\midrule
Limit&2.238334344&6.335393176\\
\bottomrule
\end{tabular}
\end{table}

A separately implemented Python checker builds the sequence by literal block
expansion and recomputes divisor sums by trial division. It agrees
with the sieve on all triples $(A(m),B(m),F(m))$ through $m=100\,000$.
It also checks the literal nested recurrence \eqref{eq:nested}
through position $16\,384$, M\"obius inversion through $100\,000$,
and the floor-sum identity at the recorded test points.
The mathematical proofs do not depend on any finite test.

For this version, $100\,000$ blocks were regenerated from the C++ source.
Every resulting integer triple agreed with the packaged data, and both
Python checkers were rerun successfully on the fresh output. The
$10^7$-block figures above belong to the recorded large run; that full
run was not repeated in this revision. The bounded reproduction used
GCC/MinGW~15.2.0 in C++17 mode, CPython~3.14.4, and
\texttt{mpmath}~1.3.0 at 70-digit working precision on 64-bit Windows~11.

A supplementary checker, \texttt{verify\_new\_properties.py}, confirms the
prime criterion for every $2\le m\le100\,000$, strict divisibility-chain
inequalities for all $1\,066\,750$ proper-divisibility pairs in that range,
and convergence of each of the $99\,999$ finite orbits starting at
$2,\ldots,100\,000$ to $2$. It also checks the exact divisor-layer total
against $F(M)$ at $M=10^3,10^4,10^5$ and records the diagnostics in
Table~\ref{tab:newdiagnostics}.

\begin{table}[ht]
\centering
\caption{Diagnostics for the additional theorems, using $M=100\,000$
blocks and $N=F(M)=149\,829\,549$ positions. The adjacent cross moment
averages $R_mR_{m+1}$ for $1\le m<M$. The last column contains the proved
limits; these finite comparisons are not error estimates.}
\label{tab:newdiagnostics}
\begin{tabular}{lrr}
\toprule
Quantity & Finite value & Limit\\
\midrule
Even-value position proportion & 0.662536593 & 0.662889556\\
Adjacent cross moment & 4.279778413 & 4.314501743\\
$\mathcal P_A(N)\log N/N$ & 0.776484010 & 0.722874129\\
$\Prob(Q_M=1)$ & 0.447931996 & 0.446760781\\
\bottomrule
\end{tabular}
\end{table}

\paragraph{Verification package and availability.}
The verification code and data accompany the preprint as a separate
reproducibility package, identified by Zenodo DOI
\href{https://doi.org/10.5281/zenodo.22999299}{\nolinkurl{10.5281/zenodo.22999299}}
\cite{zenodoverification}.
The package contains \texttt{generate.cpp}, \texttt{verify.py},
\texttt{verify\_new\_properties.py}, the generated JSON reports,
$10\,000$-term lists for the three sequences, and the recorded integer
check data through the scopes stated above. Its README gives the
compiler flags, Python commands, and tested software versions; the
machine-readable reproduction report records the scope and file hashes.
The arXiv source bundle is separate and contains only the manuscript
sources, with an inline bibliography.

\section*{AI-Disclosure}
OpenAI's ChatGPT, accessed through Codex, assisted with mathematical
exploration, the drafting and revision of proofs, literature searches,
language editing, LaTeX preparation, and the development of verification
programs. Responsibility for the mathematical claims, references, and
final presentation rests with the author. AI-assisted checks and finite
computations do not constitute independent peer review.

\end{document}